\documentclass[10pt]{article}
\usepackage{amsthm, amsmath, amssymb,enumerate}
\pdfoutput=1

\newtheorem{theorem}{Theorem}[section]
\newtheorem{corollary}[theorem]{Corollary}
\newtheorem{lemma}[theorem]{Lemma}
\newtheorem{proposition}[theorem]{Proposition}
\newtheorem{remark}[theorem]{Remark}
\newtheorem{question}[theorem]{Question}
\usepackage[pdftex]{graphicx,color}
\usepackage{pinlabel}
\input{xy}
\xyoption{all}
\usepackage{tikz}

\newcommand{\vol}{\mathrm{vol}}

\newcommand{\lmod}{\! \setminus \!}
\DeclareMathOperator{\Vol}{Vol}
\DeclareMathOperator{\SO}{SO}
\DeclareMathOperator{\OO}{O}

\DeclareMathOperator{\diam}{diam}
\DeclareMathOperator{\stab}{stab}
\DeclareMathOperator{\PSL}{PSL}

\DeclareMathOperator{\Dom}{Dom}

\title{Strict contractions and exotic $\SO_0(d,1)$ quotients.}

\author{Grant S. Lakeland, Christopher J. Leininger}

\begin{document}

\maketitle

\begin{abstract} 
For $d \leq 4$, we describe an elementary construction of nonzero degree, strict contractions between closed, oriented hyperbolic $d$--orbifolds.  Appealing to work of Gu\'eritaud--Kassel and Tholozan, these examples determine exotic quotients of $\SO_0(d,1)$.
\end{abstract}

\section{Introduction} \label{S:introduction}

The Schwarz-Pick Theorem (see \cite{Pick,Osserman}) implies that a holomorphic branched covering map between closed hyperbolic surfaces is $K$--Lipschitz for some $K \leq 1$, with $K < 1$ if the branch locus is nonempty.  For any branched covering between closed surfaces and any choice of hyperbolic metric on the target, uniformizing the pulled-back complex structure produces a hyperbolic metric on the domain to which the Schwarz-Pick theorem applies, providing an abundance of positive degree, strict contractions between compact hyperbolic surfaces.  The following question is motivated by a question of Tholozan~\cite{TholozanVol} (see \S~\ref{S:strictly dominating intro}).

\begin{question} \label{Q:dominations exist?} Can one construct nonzero degree strict contractions between hyperbolic $d$--manifolds/orbifolds for $d \geq 3$?
\end{question}

Ian Agol described an example in dimension $3$ (see \cite{AgolOverflow} and Theorem~\ref{T:Agol} below).  Here we describe a construction for infinite families of degree $1$, $K$--Lipschitz maps between hyperbolic $d$--orbifolds, for $d \leq 4$, in which $K < 1$ can be made arbitrarily small, as well as examples for which the optimal $K$ is arbitrarily close to $1$.  Given $\bar \Gamma \leq \OO_0(d,1)$, a cocompact right-angled hyperbolic reflection group, we call the index two, orientation preserving subgroup $\Gamma = \bar \Gamma \cap \SO_0(d,1)$ the {\em right-angled rotation (sub)group} (see \S\ref{S:RA polyhedra}).  For any discrete subgroup $\Gamma < \SO_0(d,1)$, we let $M_\Gamma = \Gamma \!\! \setminus \mathbb H^d$. 

\begin{theorem} \label{T:main RACG} For any two cocompact right-angled rotation groups $\Gamma,\Gamma' \leq \SO_0(d,1)$, with $d \leq 4$, there exist finite index subgroups $\{\Gamma_n \leq \Gamma\}_{n=1}^\infty$ and degree $1$, $K_n$--lipschitz maps
\[ f_n \colon M_{\Gamma_n} \to M_{\Gamma'}\]
with $K_n \to 0$ as $n \to \infty$.
\end{theorem}

To produce examples where the optimal constant is arbitrarily close to $1$, we must also take finite index subgroups of the target.

\begin{theorem} \label{T:main RACG2} For any two cocompact right-angled rotation groups $\Gamma,\Gamma' \leq \SO_0(d,1)$, with $d \leq 4$, there exist finite index subgroups $\{\Gamma_n \leq \Gamma\}_{n=1}^\infty$, $\{\Gamma_n' \leq \Gamma'\}_{n=1}^\infty$, and degree $1$, $K_n$--Lipschitz maps
\[ f_n \colon M_{\Gamma_n} \to M_{\Gamma_n'} \]
with $K_n < 1$ for all $n \geq 1$.  Furthermore, for all $n$ there exists a hyperbolic element $\gamma_n \in \Gamma_n$ such that the ratio of translation lengths $\frac{\ell(f_{n*}(\gamma_n))}{\ell(\gamma_n)}$ tends to $1$ as $n$ tends to infinity.
\end{theorem}

We note that a $K$--Lipschitz map $f \colon M_\Gamma \to M_{\Gamma'}$ has the property that $\frac{\ell(f_*(\gamma))}{\ell(\gamma)} \leq K$, and hence the optimal Lipschitz constant must tend to $1$ for the sequence produced by Theorem~\ref{T:main RACG2}.  Passing to appropriate torsion-free, finite index subgroups, we also obtain nonzero degree strict contractions between hyperbolic manifolds.

The restriction to dimension at most $4$ is necessary simply because these are the only dimensions in which examples of cocompact right-angled reflection groups exist (see \cite{pogorelov,taiyo} in dimension $3$ and \cite{Davis} in dimension $4$ for examples, and \cite{potvin,vinberg,dufour} for nonexistence in higher dimensions.).

Our construction is very elementary, but there are many more intricate constructions of nonzero degree maps between hyperbolic $3$--manifolds (see e.g.~\cite{WangDegree}).  In particular,  for any two closed hyperbolic $3$--manifolds $M_\Gamma$ and $M_{\Gamma'}$, Hongbin Sun \cite{hongbin} has a very robust method for constructing degree $2$ maps from some finite sheeted cover of $M_\Gamma$ to $M_{\Gamma'}$.  In recent work, Sun and Liu \cite{LiuSun} have refined this construction further to produce strict contractions of degree $1$.  The techniques in both cases involve quite a bit more technical machinery, including the work of Kahn-Markovic \cite{kahnmarkovic} on constructing nearly totally geodesic surfaces in hyperbolic $3$--manifolds and that of Agol and Wise proving that hyperbolic $3$--manifold groups are LERF (see \cite{AgolVhaken} and \cite{AgolGrovesManning}).  We would also like to note that in earlier work, Gaifullin \cite{gaifullin} produced nonzero degree maps from finite sheeted covers of orbifolds obtained from right-angled rotation groups onto manifolds by different techniques, though that construction provided no explicit control over Lipschitz constants.


\subsection{Strictly dominating representations} \label{S:strictly dominating intro}

Suppose $\Gamma \leq \SO_0(d,1)$ is a lattice, $j \colon \Gamma \to \SO_0(d,1)$ the inclusion, $\rho \colon \Gamma \to \SO_0(d,1)$ is another representation, and write $j \times \rho(\Gamma) < \SO_0(d,1) \times \SO_0(d,1)$.  For $K < 1$, one says that {\em $j$ strictly $K$--dominates $\rho$} if there exists a $(j,\rho)$--equivariant map $\widetilde f \colon \mathbb H^d \to \mathbb H^d$ which is $K$--Lipschitz.    If $\Gamma$ and $\rho(\Gamma)$ are both cocompact lattices, then an equivariant $K$--Lipschitz map $\widetilde f \colon \mathbb H^d \to \mathbb H^d$ descends to a $K$--Lipschitz map $f \colon M_\Gamma \to M_{\rho(\Gamma)}$ (and any $K$--Lipschitz map $f$ lifts to an equivariant $K$--Lipschitz map $\widetilde f$).

The following is a special case of a result of Gu\'eritaud--Kassel \cite{GueritaudKassel}.   Note that $\SO_0(d,1) \times \SO_0(d,1)$ acts on $\SO_0(d,1)$ on the left by  $(a,b) \cdot x = axb^{-1}$.

\begin{theorem} [Gu\'eritaud--Kassel]  \label{T:G-K} Suppose $j \colon \Gamma \to \SO_0(d,1)$ is the inclusion of a lattice and $\rho \colon \Gamma \to \SO_0(d,1)$ is a representation.  Then the action of $\Gamma$ on $\SO_0(d,1)$ from the inclusion $j \times \rho(\Gamma) \subset \SO_0(d,1) \times \SO_0(d,1)$ is properly discontinuous if and only if $\rho$ is strictly dominated by $j$.
\end{theorem}
The case $d=2$ is due to Kassel \cite{KasselThesis}.  Theorem~\ref{T:G-K} has been generalized by Gu\'eritaud--Guichard-Kassel--Wienhard \cite{GGKW} and in a different direction by Danciger--Gu\'eritaud--Kassel \cite{DanGuerKassRACG}.  There are numerous examples of properly discontinuous actions on $\SO_0(d,1)$; see \cite{GoldmanNonstandard,Ghys,KobayashiDef,Salein,KasselThesis,GueritaudKassel,GuerKassWolff,DeroinTholozan,TholozanVol,DanGuerKass}.

For a lattice $\Gamma < \SO_0(d,1)$, the quotients $j \times \rho(\Gamma)\lmod \SO_0(d,1)$ as in the theorem are $\SO(n)$--bundles over $M_\Gamma$; see \cite{GueritaudKassel}.
There are three associated volumes in this situation: $\Vol(j \times \rho)$, the volume of the quotient $j \times \rho(\Gamma)\lmod \SO_0(d,1)$ with respect to a (suitably normalized) invariant volume form on $\SO_0(d,1)$; $\Vol(j)$, the volume of the quotient $M_\Gamma = j(\Gamma) \lmod \mathbb H^d$; and $\Vol(\rho)$, the volume of the representation $\rho$, defined by the integral
\[ \Vol(\rho) = \int_{M_\Gamma} \widetilde f^* d\vol_{\mathbb H} \]
where $\vol_{\mathbb H}$ is the hyperbolic volume form on $\mathbb H^d$, $\widetilde f \colon \mathbb H^d \to \mathbb H^d$ is a $(j,\rho)$--equivariant, piecewise smooth map, and $\widetilde f^* d \vol_{\mathbb H}$ is the pull-back of $\vol_{\mathbb H}$ by $\widetilde f$ (pushed down to $M_\Gamma$).

In \cite{TholozanVol} Tholozan proved the following theorem relating these volumes.
\begin{theorem}[Tholozan] \label{T:Tholozan volume}
If $\Gamma < \SO_0(d,1)$ is a lattice and the inclusion $j$ strictly dominates a representation $\rho$, then 
\[ \Vol(j \times \rho) = {\bf V}_d(\Vol(j) + (-1)^d \Vol(\rho))\]
where ${\bf V}_d$ is the volume of $\SO(d)$ (suitably normalized).  Moreover, the function $\rho \mapsto \Vol(j \times \rho)$ is locally constant on the space $\Dom(\Gamma,\SO_0(d,1))$ of representations $\rho$ of $\Gamma$ strictly dominated by $j$, except when $n =2$ and $\Gamma$ is noncocompact.
\end{theorem}

When $\Gamma < \SO_0(2,1)$ is a lattice, the possible values of $\Vol(j \times \rho)$ for a representation $\rho \in \Dom(\Gamma,\SO_0(2,1))$ are well understood.  When $\Gamma < \SO_0(2,1)$ is a torsion-free, cocompact lattice, the work of Salein \cite{Salein} and Tholozan \cite{TholozanVol} shows that the image is exactly the finite set $\{ 4\pi^2 k \mid 1 \leq k < -2\chi(\Gamma \! \setminus \! \mathbb H^2) \}$.  If $\Gamma < \SO_0(2,1)$ is a torsion-free, noncocompact lattice, then Tholozan \cite{TholozanVol} showed that the images is the entire open interval $(0,-8\pi^2\chi(\Gamma \! \setminus \! \mathbb H^2))$.

For $d \geq 3$ and $\Gamma < \SO_0(d,1)$, the situation is much less clear.  The only known value of $\Vol(j \times \rho)$ is given by ${\bf V}_d \Vol(j)$ since all known examples of $\rho \in \Dom(\Gamma,\SO_0(d,1))$ are deformations of the trivial representation (and sufficiently small, nontrivial deformations of the trivial representation are indeed in $\Dom(\Gamma,\SO_0(d,1))$ when $\Gamma$ is cocompact; see Ghys \cite{Ghys}, Kobayashi \cite{KobayashiDef}, and Gu\'eritaud--Kassel \cite{GueritaudKassel}).  Question 5.2 of \cite{TholozanVol} asks if there are any others.  Specifically, it asks whether there are any representations $\rho \colon \Gamma \to \SO_0(d,1)$ strictly dominated by $j$ with $\Vol(j \times \rho) \neq {\bf V}_d \Vol(j)$, when $d \geq 3$.
As a corollary of Theorem~\ref{T:main RACG} (or Theorem~\ref{T:main RACG2}) and Theorem~\ref{T:Tholozan volume}, we answer this question in the affirmative, for $d = 3$ and $4$.

In the special case that $\Gamma$ and $\rho(\Gamma)$ are both cocompact lattices and there exists contracting map $f \colon M_\Gamma \to M_{\rho(\Gamma)}$, then $\Vol(\rho) = \deg(f) \Vol(N_\rho)$.  By Theorem~\ref{T:G-K}, $j \times \rho(\Gamma)$ acts properly discontinuously, and the volume formula from Theorem~\ref{T:Tholozan volume} becomes
\[ \Vol(j \times \rho) = {\bf V}_d(\Vol(M_\Gamma) + (-1)^d \deg(f)\Vol(M_{\rho(\Gamma)})).\]
Consequently, we immediately obtained the following.
\begin{corollary} \label{C:volume comp 1} Let $f_n \colon M_{\Gamma_n} \to M_{\Gamma'}$ be the strict $K_n$--domination from Theorem~\ref{T:main RACG}.  Then $f_{n*} \colon \Gamma_n \to \Gamma' < \SO_0(d,1)$ is strictly dominated by the inclusion $j \colon \Gamma_n \to \SO_0(d,1)$ and
\begin{eqnarray*} 
\Vol(j \times f_{n*}) &=& {\bf V}_d(\Vol(M_{\Gamma_n}) +(-1)^d \Vol(M_{\Gamma'}))\\
& = & {\bf V}_d([\Gamma:\Gamma_n] \Vol(M_\Gamma) + (-1)^d \Vol(M_{\Gamma'})).
\end{eqnarray*}
A similar statement holds for $f_n \colon M_{\Gamma_n} \to M_{\Gamma_n'}$ from Theorem~\ref{T:main RACG2}
\end{corollary}
We will see that as another consequence of Theorem~\ref{T:main RACG}, the values of the function $\rho \mapsto \Vol(j \times \rho)$  on $\Dom(\Gamma,\SO_0(d,1))$ can be any arbitrarily large (finite) number in dimension $d =3,4$; see Corollary~\ref{C:arbitrarily large image}.

The quotients of $\SO_0(2,1)$ are particularly interesting because $\SO_0(2,1)$ can be identified with {\em anti-de Sitter $3$--space} and $\SO_0(2,1) \times \SO_0(2,1)$ the identity component of its isometry group.  Thus the quotients are anti-de Sitter $3$--manifolds.  In fact, up to finite covers, all anti-de Sitter $3$--manifolds are obtained in this way; see Klingler \cite{Klingler} and Kulkarni-Raymond \cite{KulkarniRaymond}.

The $3$--dimensional case is also quite interesting because of the isomorphism $\SO_0(3,1) \cong \PSL(2,\mathbb C)$.  In this case, the Killing form is a bi-invariant {\em holomorphic Riemannian metric of constant negative curvature};  see Ghys \cite{Ghys}.  Therefore, the quotients $j \times \rho(\Gamma)\lmod \SO_0(3,1)$ inherit such holomorphic Riemannian metrics; see the work of Dumitrescu \cite{Dumitrescu} and Dumitrescu-Zeghib \cite{DumitrescuZeghib} for classification results for these structures in low dimensions.  Ghys \cite{Ghys} studied such quotients in the case that $\rho$ is a small deformation of the trivial representation, proving that these were precisely the small deformations of the complex structure.  In this context, the $3$--dimensional examples we have presented here could be viewed as {\em exotic} $\SO_0(3,1)$--quotients (and similarly exotic $\SO_0(4,1)$--quotients for the $d=4$ case).

As mentioned above, Danciger--Gu\'eritaud--Kassel \cite{DanGuerKassRACG} prove a generalization of Theorem~\ref{T:G-K}.  Specifically, they consider the pseudo-Riemannian case $\mathbb H^{p,q}$ associated to $\SO(p,q+1)$, and prove that proper discontinuity only requires strict contraction in the space-like directions.  Using this they are able to construct new (noncocompact) properly discontinuous actions of any right-angled Coxeter group via deformations of the standard action (as well as some interesting, explicit constructions of deformations of hyperbolic reflection groups in low dimensions).  Using an infinitesimal form of their condition, they also provide a whole new class of groups acting properly discontinuous as affine transformations (see~\cite{Margulis}, \cite{Drumm}, \cite{GLM}, \cite{CDG}, \cite{ChoiGoldman}, \cite{Smilga} for more on these types of affine actions).


\subsection{Tetrahedral groups}

Agol's original example is obtained from a specific pair of tetrahedral groups; rotation subgroups of groups generated by reflections in the faces of a pair of tetrahedra.  In this case one can explicitly construct a $K$--Lipschitz diffeomorphism between the tetrahedra which induces a homomorphism of reflection groups (and consequently a degree $1$ map between quotient orbifolds of the rotation subgroups).  As Agol's examples have not appeared in print, we recall his construction in Section~\ref{S:tetrahedron}; see Theorem~\ref{T:Agol}.

In fact, a slight variation of Agol's example can be used to produce examples of strictly dominated representations $\rho \colon \Gamma \to \SO_0(3,1)$ in which $\Gamma$ is a noncocompact lattice; see Theorem~\ref{T:noncompact examples}.  In this case, Gu\'eritaud and Kassel \cite{GueritaudKassel} observed that $\rho$ must be {\em cusp deteriorating}, meaning that for any parabolic $\gamma \in \Gamma$, $\rho(\gamma)$ must be elliptic (or trivial).  This is indeed the case in our examples.   Using an explicit form of this cusp deterioration, Thurston's {\em hyperbolic Dehn filling}, together with the examples from Theorem~\ref{T:noncompact examples}, provide another robust source of nonzero degree, $K$--Lipschitz maps (with $K <1$) between compact quotients of $\mathbb H^3$.  See Section~\ref{S:tetrahedron} and Theorem~\ref{T:dehn filling examples} for a more precise statement and details.

\noindent {\bf Plan of the paper.}  In Section~\ref{S:contraction} we describe three well-known methods of constructing contractions $\mathbb H^d \to \mathbb H^d$.  These will be the building blocks for our constructions.   Then in Section~\ref{S:RA polyhedra} we prove some preliminary facts about right-angled hyperbolic polyhedra, their tilings of hyperbolic $3$--space, and the associated ``hulls''.  This has its roots in Scott's work \cite{ScottSep}, and similar techniques were applied by Agol-Long-Reid \cite{AgolLongReid}.  Next we prove a general criterion which ensures the existence of a homomorphism between right-angled rotation groups and equivariant contractions $\mathbb H^d \to \mathbb H^d$, so that the descent to the quotients are degree $1$ contractions (see Theorem~\ref{T:main 1 simpler}).   Theorems~\ref{T:main RACG} and \ref{T:main RACG2} are applications of this criterion.  In the final Section~\ref{S:tetrahedron}, we explain Agol's original example (Theorem~\ref{T:Agol}), and a variant on this for noncompact reflection groups (Theorem~\ref{T:noncompact examples}), and how to apply Thurston's hyperbolic Dehn filling to obtain more compact-to-compact examples (Theorem~\ref{T:dehn filling examples}).

\bigskip

\noindent {\bf Acknowledgements.}  The authors would like to thank Ian Agol, Fanny Kassel, Bruno Martelli, Alan Reid, Hongbin Sun, and Nicolas Tholozan for suggestions and helpful conversations.  The second author was partially supported by NSF grant DMS-1510034.

\section{Contractions} \label{S:contraction}

Here we briefly describe three methods for constructing contractions $\mathbb H^d \to \mathbb H^d$.   The contractions in our proof will use these as building blocks.

\subsection{Projective maps} \label{S:projective maps}

For the first construction, we consider the projective model of $\mathbb H^d$,
\[ \mathbb K^d = \{ x \in \mathbb R^d \mid |x| < 1 \},\]
where $|x|$ is the Euclidean norm on $\mathbb R^d$.  The hyperbolic metric on $\mathbb K^d$ is given by one-half the Hilbert metric: 
\[ d(x,y) = \frac{1}{2} \log\left( \frac{|x-y'||y -x']}{|x-x'||y-y'|} \right), \] 
for all $x \neq y \in \mathbb K^d$, where $x',y'$ are endpoints on the unit sphere of the unique Euclidean straight line segment containing $x$ and $y$, such that the vector $y'-x'$ is a positive multiple of $y-x$.

\begin{figure}[htb]
\begin{center}
\begin{tikzpicture}[xscale=1.5,yscale=1.5]
\draw (0,0) circle [radius=1];
\draw [fill] ({1/sqrt(2)},{1/sqrt(2)}) circle [radius=.02];
\draw [fill] (-1,0) circle [radius=.02];
\draw [fill] ({-1+(1/sqrt(2)+1)/4},{(1/sqrt(2))/4}) circle [radius=.02];
\draw [fill] ({-1+(3/sqrt(2)+3)/4},{(3/sqrt(2))/4}) circle [radius=.02];
\draw (-1,0) -- ({1/sqrt(2)},{1/sqrt(2)});
\node at ({-1+(1/sqrt(2)+1)/4+.1},{(1/sqrt(2))/4-.08}) {$x$};
\node at ({-1+(3/sqrt(2)+3)/4+.1},{(3/sqrt(2))/4-.08}) {$y$};
\node [left] at (-1,0) {$x'$};
\node at ({1/sqrt(2)+.2},{1/sqrt(2)+.06}) {$y'$};
\end{tikzpicture}
\caption{\label{F:hilbert metric picture} The segment $[x',y']$ used to define the Hilbert metric.}
\end{center}
\end{figure}
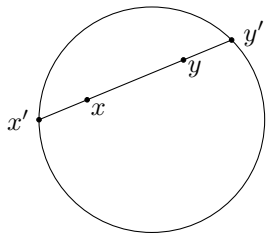

The following is a special case of a classical result of Birkhoff \cite{Birkhoff}.

\begin{lemma} [Birkhoff] \label{L:birkhoff} Suppose $f \colon \mathbb R^d \to \mathbb R^d$ is a linear transformation such that $f(\mathbb K^d) \subset \mathbb K^d$, and $\diam_{\mathbb H}(f(\mathbb K^d)) = \delta < \infty$.  Then $f$ is a $K$--Lipschitz map, with $K = \tanh(\delta/2) < 1$. 
\end{lemma}
\begin{remark} More generally, if we consider $\mathbb R^d \subset {\mathbb RP}^d$, then the conclusion of Lemma~\ref{L:birkhoff} remains true for any projective transformation $f \colon {\mathbb RP}^d \to {\mathbb RP}^d$ satisfying the hypotheses, though we will only use it in the case that $f$ is linear.
\end{remark}

\begin{corollary} \label{C:explicit birkhoff}  For any $\lambda < 1$, if $f \colon \mathbb K^d \to \mathbb K^d$ is a linear map and $f(\mathbb K^d)$ is contained in the ball of Euclidean radius $\lambda$ centered at the origin, then $f$ is $\lambda$--Lipschitz map.
\end{corollary}
\begin{proof} Using the formula above for the hyperbolic metric, we can compute that the diameter of $f(\mathbb K^d)$ is $\delta = \log\left( \frac{1+\lambda}{1-\lambda} \right)$.  By Lemma~\ref{L:birkhoff}, $f$ is $K$--Lipschitz, where $K = \tanh(\delta/2)$.  An elementary calculation then shows $K = \tanh(\delta/2) = \lambda$, as required.
\end{proof}

\subsection{Spherical coordinate maps} \label{S:spherical coordinate maps}

Next we use spherical coordinates to construct Lipschitz maps between concentric balls in $\mathbb H^d$.  Specifically, fix a basepoint $x$ in $\mathbb H^d$ and let $B_r = B_r(x)$ denote the closed ball of radius $r> 0$ centered at $x$.  We let $(t,\theta)$ denote ``spherical coordinates" on $B_r$, where $t \in [0,r]$ and $\theta \in S^{d-1}$, the unit sphere in the tangent space to the basepoint.  The paths $\theta = $ {\em constant} are geodesics called {\em radial geodesics}, and the spheres $t =$ {\em constant}, denoted $S_t^{d-1}$, and called {\em normal spheres}, are orthogonal to the radial geodesics.  The hyperbolic metric is easily described in these spherical coordinates as an orthogonal direct sum $dt^2 + \sinh^2(t) ds^2_{S^{d-1}}$ where $ds^2_{S^{d-1}}$ is the round metric from the (Euclidean) metric on the tangent space at $x$.

For any $R > r$, let $\psi = \psi_{r,R} \colon B_R \to B_r$ be the map given by $\psi(t,\theta) = (rt/R,\theta)$ in spherical coordinates.  That is, the map is simply a homothety by a factor $r/R$ applied to each radial geodesic.
\begin{lemma}  \label{L:spherical contraction} For all $R > r$, the map $\psi =\psi_{r,R}$ is $r/R$--Lipschitz.  Furthermore, the restriction $\psi|_{S_R^{d-1}} \colon S_R^{d-1} \to S_r^{d-1}$ is $e^{r-R}$--Lipschitz.
\end{lemma}
\begin{proof} We start by proving the first statement.  Set $\lambda = r/R$, so that $\psi(t,\theta) = (\lambda t,\theta)$.  Then we are required to prove that $\psi$ is $\lambda$--Lipschitz.  First observe that by construction, $\psi$ is $\lambda$--Lipschitz on each radial geodesic.  It also sends normal spheres to normal spheres, and thus it suffices to prove that for all $t \in (0,R]$, the restriction of $\psi|_{S_t^{d-1}} \colon S_t^{d-1} \to S_{\lambda t}^{d-1}$ is also $\lambda$--Lipschitz.

Observe that the restriction scales lengths precisely by $\sinh(\lambda t)/\sinh(t)$, and so is $\sinh(\lambda t)/\sinh(t)$--Lipschitz map.  To complete the proof, we must show that $\sinh(\lambda t)/\sinh(t) \leq \lambda$, or equivalently, $\sinh(\lambda t) \leq \lambda \sinh(t)$.  This follows immediately from the fact that $\sinh$ is convex on $[0,\infty)$:
\[ \sinh(\lambda t) = \sinh((1-\lambda)0 + \lambda t) \leq (1-\lambda)\sinh(0) + \lambda \sinh(t) = \lambda \sinh(t).\]
This completes the proof of the first statement.

For the second statement, observe that we have already pointed out that $\psi|_{S_R^{d-1}}$ is $\sinh(r)/\sinh(R)$--Lipschitz.  Therefore, since $0 < r < R$, we have the following bound, which completes the proof:

$\hspace{2cm} \displaystyle{ \frac{\sinh(r)}{\sinh(R)} = \frac{e^r}{e^R} \frac{1-e^{-2r}}{1-e^{-2R}} \leq e^{r-R}.}$ \end{proof}

\subsection{Closest point projections} \label{S:closest point projections}

A third construction we will take advantage of is the {\em closest point projection}.   While this is not a strict contraction, it will be useful as an intermediate map.

Given a closed convex set $X \subset \mathbb H^d$, let $\pi \colon \mathbb H^d \to X$ be the closest point project, so that for all $x \in \mathbb H^d$, $\pi(x)$ is the closest point of $X$ to $x$,
\[ d(x,\pi(x)) = \inf_{x' \in X} d(x,x').\]
Convexity of $X$ implies that there exists a unique such closest point $\pi(x) \in X$, and that $\pi$ is continuous.  Since $\pi$ is the identity on $X$, it cannot be a contraction.  Nonetheless, the following well-known property is sufficient for our purposes (see \cite{EM}, Lemma 1.3.4, for example).
\begin{lemma} \label{L:exponential contraction}
For any closed convex set $X \subset \mathbb H^d$,  the closest point projection $\pi \colon \mathbb H^d \to X$ is $1$--Lipschitz.
\end{lemma}

\section{Right-angled polyhedra construction} \label{S:RA polyhedra}

\subsection{Right-angled reflection groups} \label{S:RACG}

A compact, right-angled hyperbolic polyhedron $Q \subseteq \mathbb H^d$ is a compact polyhedron such that every dihedral angle is $\pi/2$.   All right-angled polyhedra that we consider will be compact, so we will drop that adjective in what follows.

Given a right-angled hyperbolic polyhedron $Q \subseteq \mathbb H^d$, let $\mathcal F(Q)$ denote the set of $d-1$--dimensional faces of $Q$.  For each $F \in \mathcal F(Q)$, let $\tau_F$ denote the hyperbolic reflection in the hyperbolic hyperplane containing $F$.  We define the {\em right-angled reflection group} $\bar \Gamma_Q \leq \OO_0(d,1)$ to be the group generated by $\{\tau_F\}_{F \in \mathcal F(Q)}$, and the {\em right-angled rotation (sub)group} to be $\Gamma_Q = \bar \Gamma_Q \cap \SO_0(d,1)$.
The group $\bar \Gamma_Q$ is a right-angled Coxeter group and has presentation
\begin{equation} \label{E:presentation} \langle \{\tau_F\}_{F \in \mathcal F(Q)} \mid \tau_F^2 = 1  \mbox{ and } [\tau_F,\tau_{F'}]=1 \mbox{ if } F \cap F' \neq \emptyset \rangle
\end{equation}
The determinant $\det \colon \bar \Gamma_Q \to \{ \pm 1 \}$ can be expressed in terms of generators by $\det(\tau_F) = -1$, for all $F \in \mathcal F(Q)$, and by definition, $\ker(\det) = \Gamma_Q$.

A fundamental domain for $\Gamma_Q$ is obtained as $Q \cup \tau_F(Q)$, for any $F \in \mathcal F(Q)$.  Consequently, $\Gamma_Q \! \lmod \mathbb H^d$ is obtained by ``doubling'' $Q$ over its boundary, and the covolume of $\Gamma_Q$ is easily computed from the volume of $Q$:
\[ \Vol(\Gamma_Q \! \lmod \mathbb H^d) = 2 \Vol(Q).\]
The $\bar \Gamma_Q$--translates of $Q$ determine a $\bar \Gamma_Q$--invariant tessellation of $\mathbb H^d$.  Let $\mathcal P(Q)$ denote the $\bar \Gamma_Q$--invariant set of hyperbolic hyperplanes containing the $d-1$--dimensional faces of the $\bar \Gamma_Q$--translates of $Q$.  The union of the hyperplanes in $\mathcal P(Q)$ is precisely the union of all $d-1$--dimensional faces of the $\bar \Gamma_Q$--translates of $Q$, and since the tessellation is locally finite, so the collection of hyperplanes in $\mathcal P(Q)$ is locally finite.

\subsection{Convex hulls} \label{S:Hulls}

For a subset $X \subset \mathbb H^d$, its convex hull is the intersection of all hyperbolic half-spaces containing $X$.  We will be interested in the {\em $Q$--convex hull of $X$}, denoted $\mathcal H_Q(X)$, defined as the intersection of the half-spaces containing $X$, bounded by hyperbolic hyperplanes in $\mathcal P(Q)$.  These are the key ingredient in Agol-Long-Reid's work \cite{AgolLongReid}, building on work of Scott \cite{ScottSep}.  We will need the following lemma which is essentially due to Agol-Long-Reid (indeed, a version of this can be viewed as a special case of the claim on page 606 of \cite{AgolLongReid}).  We sketch a proof for completeness.

\begin{lemma} \label{L:finite Q-hull} For any compact right-angled hyperbolic polyhedron $Q$ there exists $R = R(Q) > 0$ with the following property.  If $X \subseteq \mathbb H^d$ is any compact, convex set, then $\mathcal H_Q(X)$ is contained in the $R$--neighborhood of $X$.  Consequently, $\mathcal H_Q(X)$ is a compact, right-angled hyperbolic polyhedron which is a union of finitely many $\bar \Gamma_Q$--translates of $Q$.
\end{lemma}
\begin{proof} The conclusion in the last sentence is clear from the first claim and so we prove the first claim. 

Observe that there exists $\epsilon > 0$ and $R_0 > 0$ so any geodesic segment of length at least $R_0$ intersects some hyperplane in $\mathcal P(Q)$ in an angle at least $\epsilon$.  To see this, note that if there were no such $\epsilon$ and $R_0$, then the limit of a sequence of counterexamples provides a geodesic with no transverse intersections with any hyperplane in $\mathcal P(Q)$, which is absurd.

From elementary hyperbolic geometry, we see that there exists $R_1> 0$ (depending on $\epsilon$) so that a geodesic segment $\sigma$ of length at least $R_1$ intersecting a hyperbolic hyperplane $P$ at an angle at least $\epsilon$ must have at least one end point with the property that the hyperbolic hyperplane $P'$ through that endpoint, orthogonal to $\sigma$ must have $P \cap P' = \emptyset$.
Set $R = R_0+R_1$.  

Suppose $x \in \mathbb H^d$ has distance greater than $R$ to $X$, let $\sigma$ be a shortest geodesic segment from $x$ to $X$, meeting $X$ at the endpoint $y$ of $\sigma$.  The initial subsegment of $\sigma$ containing $x$ of length $R_0$ meets some hyperplane $P \in \mathcal P(Q)$ at a point $z$ and making an angle greater than $\epsilon$.  The subsegment of $\sigma$ from the point of intersection $z$ to $y \in X$ has length at least $R_1$.  Therefore the hyperplane $P'$ through that point in $X$, orthogonal to $\sigma$, must be disjoint from $P$.  Since $\sigma$ connects $x$ to $y$, it follows from convexity that $X$ is contained in the half-space determined by $P'$ (not containing $x$).  Consequently, $X$ is in the half-space bounded by $P$, not containing $x$, and so $x \not \in \mathcal H_Q(X)$, completing the proof.  See Figure~\ref{F:Qhull}.\end{proof}

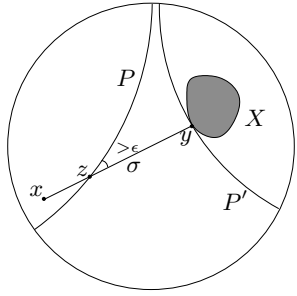
\begin{figure} [htb]
\labellist
\small\hair 2pt
\pinlabel $x$ [b] at 15 45
\pinlabel $y$ [b] at 90 70
\pinlabel $z$ [b] at 38 55
\pinlabel $\sigma$ [b] at 64 57
\pinlabel $X$ [b] at 125 80
\pinlabel $P$ [b] at 60 100
\pinlabel $P'$ [b] at 115 38
\pinlabel $^{>\epsilon}$ [b] at 61 62
\endlabellist
\begin{center}
\includegraphics[scale=0.75]{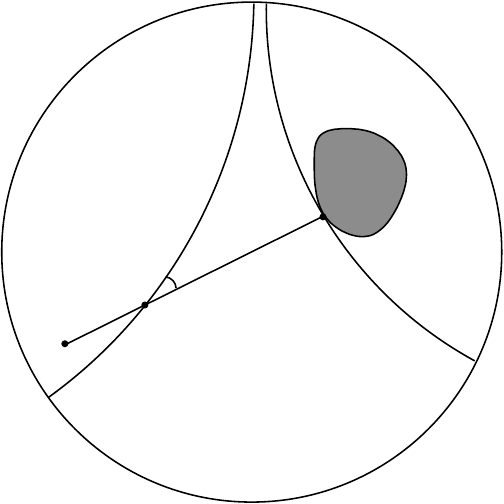}
\caption{\label{F:Qhull} Cartoon of Lemma~\ref{L:finite Q-hull}.}
\end{center}
\end{figure}

\begin{lemma} \label{L:bounded diameter faces} Given a right-angled polyhedron $Q$, let $R = R(Q) > 0$ be the constant from Lemma~\ref{L:finite Q-hull} and $D = D(Q) = 2R+4$.  If $X \subseteq \mathbb H^d$ is any compact, convex set and $r \geq 1$, then the diameter of any face of $\mathcal H_Q(N_r(X))$ is at most $D$.
\end{lemma}
Appealing to Lemma~\ref{L:finite Q-hull}, Lemma~\ref{L:bounded diameter faces} is essentially asserting a(n explicit) bound on the length of a geodesic that lies in the ``shell'' between $N_r(X)$ and $N_{r+R}(X)$.
\begin{proof}  From Lemma~\ref{L:finite Q-hull} we have
\[ \mathcal H_Q(N_r(X)) \subseteq N_R(N_r(X)) = N_{R+r}(X),\]
for all $r > 0$.  Let $\sigma$ be a segment in some face of $\mathcal H_Q(N_r(X))$ whose length is equal to the maximal diameter of any face of $\mathcal H_Q(N_r(X))$.  Let $x \in \sigma$ and $y \in X$ realize the minimum distance between $\sigma$ and $X$.  Subdividing $\sigma$ at $x$, let $\sigma_0 \subseteq \sigma$ be the longer of the two resulting segments, and let $z$ denote its other endpoint (so $d(x,z) = \ell(\sigma_0) \geq \ell(\sigma)/2$).  Then $x,y,z$ are the vertices of a hyperbolic triangle with non-acute angle at $x$.  A computation for the extremal case in $\mathbb H^2$ (essentially the same computation to prove that triangles are $\log(1+\sqrt{2})$--slim) shows that
\[ d(y,x) + d(x,z) \leq d(y,z) + \log(1+\sqrt{2}) < d(y,z) + 1.\]

Next, let $\Sigma$ be the hyperbolic hyperplane through $y$, orthogonal to the geodesic from $x$ to $y$.  Then $X$ and $z$ lie on opposite side of $\Sigma$ so that $d(\Sigma,z) \leq d(X,z)$.  On the other hand, since $d(y,x) \geq r \geq 1 > \log(1+\sqrt{2})$, a computation in $\mathbb H^2$ shows that $d(y,z) \leq d(\Sigma,z) + \log(1+\sqrt{2}) < d(\Sigma,z) + 1$ (see Figure~\ref{F:projection_bound} for the extreme case).

\begin{figure}[htb]
\begin{center}
\begin{tikzpicture}[xscale=2,yscale=2]
\draw[samples=40, domain=-.5*pi:-.25*pi] plot ({cos(\x r)},{sqrt(2)+sin(\x r)});
\draw[samples=40, domain=.75*pi:pi] plot ({sqrt(2)+cos(\x r)},{sin(\x r)});
\draw (-1,0) -- (1,0);
\draw (0,0) -- (0,{sqrt(2)-1});
\draw (0,0) circle [radius=1];
\draw [fill] (0,0) circle [radius=.02];
\draw [fill] (0,{sqrt(2)-1}) circle [radius=.02];
\draw [fill] ({sqrt(2)-1},0) circle [radius=.02];
\draw [fill] ({1/sqrt(2)-.01},{1/sqrt(2)-.01}) circle [radius=.02];
\node at (-.75,-.1) {$\Sigma$};
\node at (.35,.6) {$\sigma_0$};
\node [above] at (0,{sqrt(2)-1}) {$x$};
\node [below] at ({sqrt(2)-1},.04) {$y'$};
\node [below] at (0,0) {$y$};
\node at ({1/sqrt(2)+.06},{1/sqrt(2)+.06}) {$z$};
\end{tikzpicture}
\caption{\label{F:projection_bound} Closest point $y' \in \Sigma$ to $z$ when $d(x,z)>>0$ and $d(y,x) = \log(1+\sqrt{2})$.}
\end{center}
\end{figure}
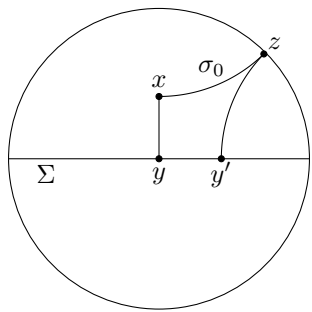

Now, since $z \in N_{r+R}(X)$, we have
\[ d(y,z) \leq d(\Sigma,z) + 1 \leq d(X,z) + 1 \leq r+R+1.\]
Combining the inequalities above, we obtain
\[ r + d(x,z) \leq d(y,x) + d(x,z) \leq d(y,z) + 1 \leq r + R + 2\]
and since $\ell(\sigma) \leq 2 \ell(\sigma_0) = 2 d(x,z)$, we have
\[ \ell(\sigma) \leq 2d(x,z) \leq 2R + 4 = D.\]
This completes the proof.
\end{proof}

\subsection{Lebesgue numbers} \label{S:Lebesgue numbers}

For any right-angled polyhedron $Q \subset \mathbb H^d$, consider $\pi \colon \mathbb H^d \to Q$, the closest point projection.  To better understand this, consider the set of hyperplanes $P_1, \ldots, P_k \in \mathcal P(Q)$ containing the $d-1$--dimensional faces of $Q$.  These divide $\mathbb H^d$ into a collection of convex regions which we denote $\mathcal R$.  There is one compact region, namely $Q$, and the rest are regions  ``opposite $Q$ across a face $\sigma$'', denoted $\mathcal R_\sigma$, where $\sigma$ is a face of any dimension.  See Figure~\ref{F:2d example} for a $2$--dimensional example.  The region $\mathcal R_\sigma$ opposite $Q$ across the face $\sigma$ is the closure of the preimage of $int(\sigma)$, the interior of $\sigma$.  For a face $\sigma$, we have
\[ \pi^{-1}(\sigma) = \bigcup_{\sigma_0 \subset \sigma} \mathcal R_{\sigma}\]
where the union is over all faces $\sigma_0$ of $\sigma$.  For a $d-1$--dimensional face $F \in \mathcal F(Q)$, $\pi^{-1}(F)$ is the half-space $\mathcal H_F$ bounded by the unique hyperplane containing $F$ and {\em not} containing $Q$.

\begin{figure} [htb]
\labellist
\small\hair 2pt
\pinlabel $F$ [b] at 50 220
\pinlabel $\mathcal H^1_F$ [b] at 205 235
\pinlabel $\mathcal R_F$ [b] at 130 240
\endlabellist
\begin{center}
\includegraphics[scale=0.6]{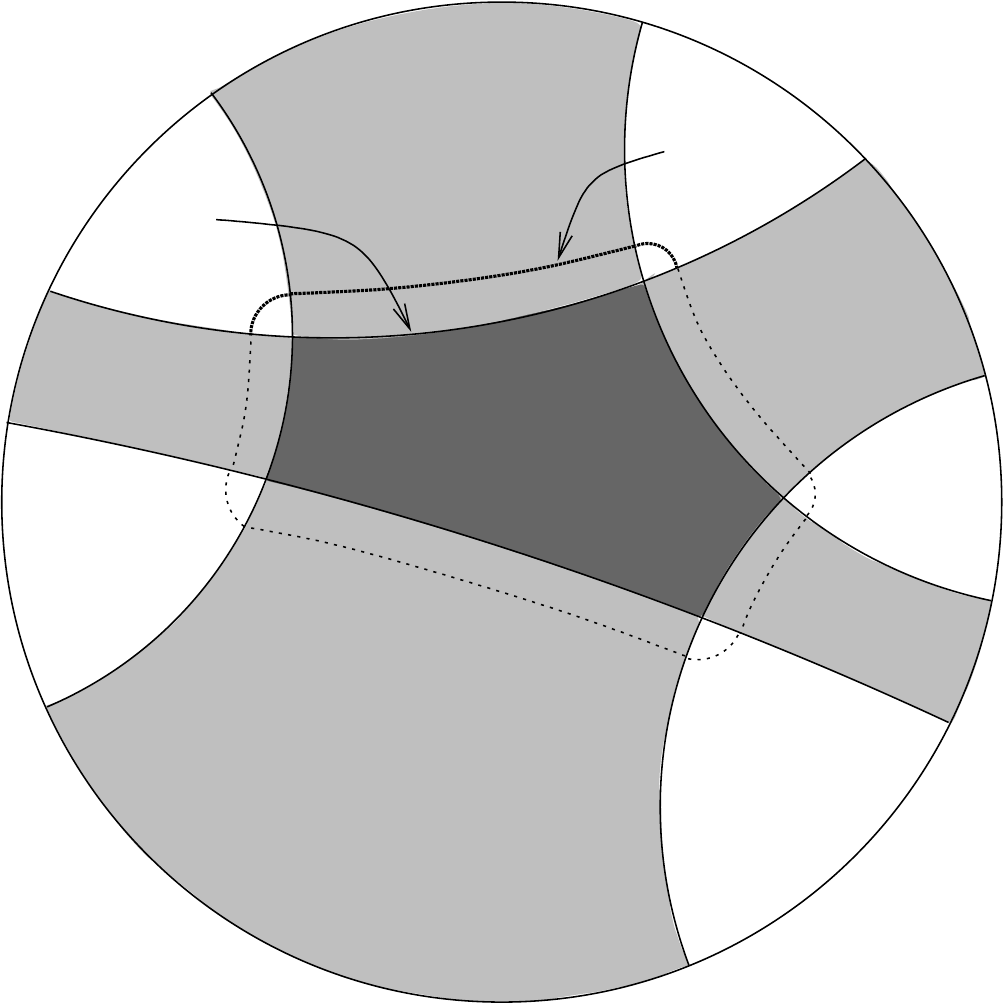}
\caption{\label{F:2d example} The dark shaded object is a right-angled pentagon $Q$.  The lines through the top-dimensional faces divide $\mathbb H^2$ into 11 regions.  The lighter shaded regions are those opposite the pentagon across the edges, and the remaining (white) regions are opposite the pentagon across the vertices.  The dotted curve represents $\partial N_1(Q)$, with the preimage of a face $F$, $\mathcal H^1_F$, represented as a darker arc.}
\end{center}
\end{figure}

We will be interested the restriction of $\pi$ to $\partial N_1(Q)$, the boundary of the $1$--neighborhood of $Q$.  With that in mind, set
\[ \mathcal H^1_F = \mathcal H_F \cap \partial N_1(Q) = \pi^{-1}(F)  \cap \partial N_1(Q).\]
The interiors of these sets form an open cover of $\partial N_1(Q)$, and we let $L(Q)$ denote at Lebesgue number of this cover.  That is, for any point $z \in \partial N_1(Q)$, the ball of radius $L(Q)$ about $z$ (in the path metric on $\partial N_1(Q)$) is contained in some set in $\{ \mathcal H_F^1 \}_{F \in \mathcal F(Q)}$.  In fact, there is a Lebesgue number $L$ that works for all polyhedra $Q$ of a given dimension.
\begin{lemma} \label{L:Lebesgue number} For any $d$ there exists $L>0$, so that for any $d$--dimensional right-angled polyhedron $Q$, $L$ is a Lebesgue number for $\{ \mathcal H_F^1 \}_{F \in \mathcal F(Q)}$,
\end{lemma}
\begin{proof} Consider a $d$--dimensional right-angled polyhedron $Q$, and for $F \in \mathcal F(Q)$, set
\[ \mathcal R^1_F = \mathcal R_F \cap \partial N_1(Q). \]
Then $\mathcal H^1_F$ is the (closed) $\frac{\pi \sinh(1)}2$--neighborhood of $\mathcal R^1_F$.
Indeed, for any point $z \in \partial \mathcal R^1_F$, the distance from $z$ to $\partial \mathcal H^1_F = \partial \mathcal H_F \cap \partial N_1(Q)$ is $\frac{\pi \sinh(1)}{2}$, with every geodesic in $\partial N_1(Q)$ from $z$ to $\partial \mathcal H^1_F$ being a quarter of a circle of radius $1$ in hyperbolic space.

Now suppose $0 < L \leq \frac{\pi \sinh(1)}{4}$, set $\Delta = \frac{\pi \sinh(1)}{2} - L$, and suppose that the (closed) $\Delta$--neighborhoods $\{N_\Delta(\mathcal R^1_F)\}_{F \in \mathcal F(Q)}$ cover $\partial N_1(Q)$.  Then given any $z \in \partial N_1(Q)$, there exists $F \in \mathcal F(Q)$ such that $z \in N_\Delta(\mathcal R^1_F)$.  Note that the distance in $\partial N_1(Q)$ from $z$ to $\partial \mathcal H^1_F$ is at least $L$, and hence the $L$--ball about $z$ in $\partial N_1(Q)$ is contained in $\mathcal H^1_F$.  Consequently $L$ is a Lebesgue number for $\{ \mathcal H_F^1 \}_{F \in \mathcal F(Q)}$.

From this, it is easy to see that $L = \frac{\pi \sinh(1)}4$ is the optimal Lebesgue number for $d = 2$.  In general, if $\sigma$ is the intersection of $n \geq 2$ of the $d-1$--dimensional faces $F_1,\ldots,F_n \in \mathcal F(Q)$, then $\mathcal R_\sigma \cap \partial N_1(Q)$ is isometric to the product of $\sigma$ (with the metric scaled by $\cosh(1)$) and a right-angled spherical $n-1$--simplex in a (Euclidean) sphere of radius $\sinh(1)$.  For each simplex slice of this product, the $n$ vertices are the points of intersection with $\mathcal R^1_{F_1},\ldots,\mathcal R^1_{F_n}$, and distance to $\mathcal R^1_{F_i}$ is precisely the distance to the $i^{th}$ vertex.  Therefore, we can take $L$ as the maximum distance from a point in such a spherical $n-1$--simplex to the union of the vertices (maximized over all $n \leq d$).   Explicitly, $L = \sinh(1) \arccos \sqrt{\tfrac{d-1}{d}}$ is the optimal Lebesgue number.
\end{proof}

\subsection{The construction} \label{S:RACG theorem}

Let $Q,Q'$ be two compact right-angled hyperbolic polyhedra.  We assume throughout (after translating $Q$ if necessary) that there is a point $x \in int(Q) \cap int(Q')$, for convenience.   For $r > 0$, let $Q_r = \mathcal H_Q(B_r(x))$.  Later we will also want to consider $Q_r' = \mathcal H_{Q'}(B_r(x))$.   Define
\[ \bar \Gamma = \bar \Gamma_Q, \, \bar \Gamma_r = \bar \Gamma_{Q_r} , \, \bar \Gamma' = \bar \Gamma_{Q'} , \, \bar \Gamma_r' = \bar \Gamma_{Q_r'},\]
and let the rotation subgroups be denoted with the bar removed.  Note that $\bar \Gamma_r < \bar \Gamma$ and $\bar \Gamma_r' < \bar \Gamma'$ are finite index subgroups.
The technical theorem we will need to prove the theorems from the introduction is the following.

\begin{theorem} \label{T:main 1 simpler} Suppose $Q,Q'$ are two compact right-angled hyperbolic polyhedra and $x \in int(Q) \cap int(Q')$ is some point.  If there exists $\delta > 1$ and $0 < K < 1$ such that
\begin{enumerate}
\item $\frac{\delta(1-K)}{K} \geq \log(\frac{D}L)$ where $D = D(Q)$ is from Lemma~\ref{L:bounded diameter faces} and $L$ is the Lebesgue number from Lemma~\ref{L:Lebesgue number}; and  
\item $N_1(Q') \subset B_\delta(x)$.
\end{enumerate}
then for $r = \frac{\delta}K$, there exists a homomorphism $\rho \colon \Gamma_r \to \Gamma'$ and a $K$--Lipschitz, $\rho$--equivariant map $\widetilde f \colon \mathbb H^d \to \mathbb H^d$.  Moreover, $\widetilde f$ descends to a degree $1$, $K$--Lipschitz map $f \colon M_{\Gamma_r} \to M_{\Gamma'}$.
\end{theorem}

For the remainder of this subsection, we fix $Q,Q'$, $x \in int(Q) \cap int(Q')$, $\delta$ and $K$, and $r = \frac{\delta}K$ as in the theorem.  We also write $B_r(x)$ simply as $B_r$, etc.
Define a map $\widetilde f^\circ \colon Q_r \to Q'$ as the composition
\[ \widetilde f^\circ = \pi' \circ \psi \circ \pi \]
where $\pi = \pi_{B_r}$ and $\pi' =\pi_{Q'}$ are the closest point projections to $B_r$ and $Q'$, respectively, and $\psi = \psi_{\delta,r} \colon B_r \to B_\delta$ is the spherical coordinates map as described in Section~\ref{S:spherical coordinate maps}.  Both $\pi$ and $\pi'$ are $1$--Lipschitz, and by Lemma~\ref{L:spherical contraction}, $\psi$ is $K$--Lipschitz.

Recall that $\mathcal F(Q)$ is the set of $d-1$--dimensional faces of $Q$.  The following easily implies Theorem~\ref{T:main 1 simpler}.
\begin{lemma} \label{L:homeo defining homo}
There exists a function $\varphi \colon \mathcal F(Q_r) \to \mathcal F(Q')$ such that for all $F \in \mathcal F(Q_r)$, $\widetilde f^\circ (F) \subset \varphi(F)$.
\end{lemma}

Assuming Lemma~\ref{L:homeo defining homo}, we prove the theorem.

\begin{proof}[Proof of Theorem~\ref{T:main 1 simpler}.]
We first claim that the function $\rho(\tau_F) = \tau_{\varphi(F)}$, defined on the generators $\{\tau_F\}_{F \in \mathcal F(Q_r)}$ extends to a homomorphism $\rho \colon \bar \Gamma_r \to \bar \Gamma'$.  For this, we must simply verify that relations of $\bar \Gamma_r$ from the presentation \eqref{E:presentation} are satisfied.  Clearly $\rho(\tau_F)^2 = \tau_{\varphi(F)}^2 = 1$, so we need only prove that the second type of relation holds.  For that, we observe that for $F,F' \in \mathcal F(Q)$, we have
\[ [\tau_F,\tau_{F'}] = 1 \quad \Leftrightarrow  \quad F \cap F' \neq \emptyset  \quad \Rightarrow   \quad \emptyset \neq \widetilde f^\circ (F) \cap \widetilde f^\circ(F') \subset \varphi(F) \cap \varphi(F')\]
\[ \hspace{4.9cm} \Rightarrow  \quad [\rho(\tau_F),\rho(\tau_{F'}] = [\tau_{\varphi(F)},\tau_{\varphi(F')}] = 1\]
as required.

Before we define the equivariant map $\widetilde f \colon \mathbb H^d \to \mathbb H^d$, we make a few observations.  First, note that for $z \in Q_r$, 
\[ \stab_{\bar \Gamma_r}(z) = \langle \tau_F \mid F \in \mathcal F(Q_r), \, z \in F \rangle.\]
Then for all $F \in \mathcal F(Q_r)$ with $z \in F$, $\widetilde f^\circ(z) \in \widetilde f^\circ(F) \subset \varphi(F)$, so $\rho(\tau_F)$ fixes $\widetilde f^\circ(z)$, and hence $\rho(\stab_{\bar \Gamma_r}(z))$ fixes $\widetilde f^\circ(z)$.  Next, we observe that $(\bar \Gamma_r \cdot z) \cap Q_r = \{z\}$ and $\bar \Gamma_r \cdot Q_r = \mathbb H^d$.  So, for any $y \in \mathbb H^d$, let $z \in Q_r$ be the unique element with $y \in \bar \Gamma_r \cdot z$.  Let $\gamma \in \bar \Gamma_r$ be such that $\gamma \cdot z = y$, and we want to define
\[ \widetilde f(y) = \rho(\gamma) \widetilde f^\circ(z).\]
To see that this is well-defined, let $\gamma' \in \bar \Gamma_r$ be any other element such that $\gamma' \cdot z = y$.  Then $\gamma^{-1} \gamma' \in \stab_{\bar \Gamma_r}(z)$, and so $\rho(\gamma^{-1}\gamma')$ fixes $\widetilde f^\circ(z)$.  Therefore, 
\[ \rho(\gamma') \widetilde f^\circ(z) = \rho(\gamma \gamma^{-1}\gamma')  \widetilde f^\circ(z) = \rho(\gamma)\rho(\gamma^{-1}\gamma')  \widetilde f^\circ(z) = \rho(\gamma)\widetilde f^\circ(z),\]
and so $\widetilde f$ is well-defined, and equivariant by construction.

For all $\gamma \in \bar \Gamma_r$, we have
\[ \widetilde f|_{\gamma \cdot Q_r} = \rho(\gamma) \widetilde f^\circ \gamma^{-1}|_{\gamma \cdot Q_r},\]
so since $\widetilde f^\circ$ is $K$--Lipschitz, so is $\rho(\gamma) \widetilde f^\circ \gamma^{-1}$, and hence also $\widetilde f$.

Note that since $\det \colon \bar \Gamma_r,\bar \Gamma' \to \{\pm 1\}$ are both defined by $\det(\tau_F) = -1$, it follows that $\det \circ \rho = \det \colon \bar \Gamma_r \to \{\pm 1\}$.   Therefore, $\rho^{-1}(\Gamma') = \Gamma_r$, and $\widetilde f$ descends to the quotient $f \colon M_{\Gamma_r} \to M_{\Gamma'}$.

The orbifolds $M_{\Gamma_r}$ and $M_{\Gamma'}$ are obtained by doubling $Q_r$ and $Q'$, respectively, over their boundaries.
As mentioned above, we can see this by taking a fundamental domain for $\Gamma_r$ to be $Q_r \cup \tau_F(Q_r)$ for some  $F \in \mathcal F(Q_r)$ and fundamental domain $Q' \cup \tau_{\varphi(F)}Q'$ for $\Gamma'$.
Then $f$ sends each copy of $Q_r$ by a degree $1$ map to a copy of $Q'$ (the map is basically $\widetilde f^\circ$), and consequently, $\deg(f) = 1$.
\end{proof}

We now turn to the
\begin{proof}[Proof of Lemma~\ref{L:homeo defining homo}.]
We write $\pi'$ as a composition $\pi' \circ \pi_0'$ where $\pi_0' = \pi_{N_1(Q')}$ is the closest point projection to $N_1(Q')$.  It is not difficult to see that this is possible (c.f.~\cite{EM}), and in any case, we could just define $\widetilde f^\circ = \pi' \circ \pi_0 \circ \psi \circ \pi$ to begin with.  We analyze $\widetilde f^\circ$ as the composition of $\pi'$ and $g = \pi_0' \circ \psi \circ \pi$.  The key property of the map $g$ is the following.

\noindent
{\bf Claim.} For any $F \in \mathcal F(Q_r)$, the diameter of $g(F)$ in $\partial N_1(Q')$ is less than $L$.
\begin{proof}
Fix $F \in \mathcal F(Q_r)$.  By Lemma~\ref{L:bounded diameter faces}, the diameter of $F$ is at most $D$.  Then by Lemma~\ref{L:exponential contraction}, $\pi(F)$ has diameter at most $D$ in $\partial B_r(x)$.  Applying Lemma~\ref{L:spherical contraction} and the first assumption of the theorem, $\psi(\pi(F))$ has diameter at most
\begin{equation} \label{E:diameter of cell structure} e^{\delta-\delta/K}D = e^{\frac{-\delta(1-K)}{K}} D \leq e^{-\log(\frac{D}L)} D = L.\end{equation}
Finally, by Lemma~\ref{L:exponential contraction} again, $g(F)= \pi_0'(\psi(\pi(F)))$ has diameter at most $L$ in $\partial N_1(Q)$, as required.
\end{proof}

Recall that $L$ is a Lebesgue number for the cover
\[ \{(\pi')^{-1}(F') \cap \partial N_1(Q_r) \}_{F' \in \mathcal F(Q')}\]
(see \S\ref{S:Lebesgue numbers}).  Therefore, by the claim, $g(F)$ is contained in at least one set $(\pi')^{-1}(F')$, for some $F' \in \mathcal F(Q')$.  Equivalently, $\widetilde f^\circ(F) \subset F'$.

Now enumerate the faces $\mathcal F(Q') = \{F_1,\ldots,F_k\}$ in any way.  For $F \in \mathcal F(Q_r)$, define $\varphi(F)= F_i$, if $\widetilde f^\circ (F) \subset F_i$ and $\widetilde f^\circ (F) \not \subset F_j$ for all $j < i$ (in particular, if $\widetilde f^\circ(F) \subset F_1$, then $g(F) = F_1$).  That is, $\varphi(F)$ is the smallest indexed face of $\mathcal F(Q')$ containing $\widetilde f^\circ(F)$.  By construction, $\widetilde f^\circ(F) \subset \varphi(F)$, as required.
\end{proof}

\subsection{Proofs of Theorems~\ref{T:main RACG} and \ref{T:main RACG2}.}

Using Theorem~\ref{T:main 1 simpler} we prove the first two theorems from in the introduction.
\begin{proof}[Proof of Theorem~\ref{T:main RACG}.]
Fix $Q,Q'$ and any $x \in int(Q) \cap int(Q')$.  Let $\delta > 1$ be such that $N_1(Q') \subset B_\delta(x)$.  Then for any $0 < K < 1$, set $r = \frac{\delta}K$.  We now note that as $K \to 0$, $r \to \infty$, and $\frac{\delta(1-K)}K \to \infty$.  Thus, taking $\{K_n\}$ to be any sequence tending to infinity (and passing to the tail of the sequence), it follows that $x,\delta,K_n$ satisfy the assumptions of Theorem~\ref{T:main 1 simpler}.  Therefore, setting $r_n = \frac{\delta}{K_n}$, $\Gamma_n = \Gamma_{r_n} < \Gamma$, and $\Gamma'  = \Gamma_{Q'}$, Theorem~\ref{T:main 1 simpler} provides $K_n$--Lipschitz, degree $1$ maps
\[ f_n \colon M_{\Gamma_n} \to M_{\Gamma'} \]
and $K_n \to 0$ as $n \to \infty$ by assumption.
\end{proof}

\begin{proof}[Proof of Theorem~\ref{T:main RACG2}.]
Again, fix $Q,Q'$ and any $x \in int(Q) \cap int(Q')$.  Let
\[ R = \max \{R(Q),R(Q') \} \]
from Lemma~\ref{L:finite Q-hull} and $D = 2R+4$, as in Lemma~\ref{L:bounded diameter faces}.  For $s \geq 1$, set
\[ \delta = \delta(s) = 1+s+R  \]
and let $K = K(s)$ with $0 < K < 1$ be defined by
\[ \frac{\delta(1-K)}{K} = \log \left( \frac{D}L \right).\]
Note that as $s \to \infty$, we have $\delta(s) \to \infty$ and $K(s) \to 1$.

Let $Q_s' =  \mathcal H_{Q'}(B_s(x))$.  By Lemma~\ref{L:finite Q-hull}, $Q_s' \subset B_{s+R}(x)$, and so 
\[ B_s(x) \subset Q_s' \subset N_1(Q_s') \subset B_\delta(x).\]
The polyhedra $Q,Q_s'$ and the constants $\delta > 1$ and $0 < K < 1$, satisfy the assumptions of Theorem~\ref{T:main 1 simpler}.  Consequently, for $r = r(s) = \frac{\delta(s)}{K(s)}$, and $Q_r = \mathcal H_Q(B_r(x))$, we have a homomorphism $\rho_s \colon \bar \Gamma_r \to \bar \Gamma_s'$ and a $K$--Lipschitz, degree $1$ map
\[ f_s \colon M_{\Gamma_r} \to M_{\Gamma_s'}\]
with $f_{s*} = \rho_s$.

Now for each $s > 1$, we want to find an element $\gamma_s \in \Gamma_r$ such that
\[ \lim_{s \to \infty} \frac{\ell(\gamma_s)}{\ell(\rho_s(\gamma_s))} = 1.\]
To do this, we first observe that $B_s(x) \subset Q_s' \subset B_\delta(x)$.  Pick a geodesic through $x$, and let $F_{s,1}',F_{s,2}' \in \mathcal F(Q')$ be faces intersected by this geodesic (there might be more than one pair, but we choose one).  The length of the segment of the geodesic between the points of intersection with $F_{s,1}'$ and $F_{s,2}'$ is between $2s$ and $2\delta$.  For $s$ very large, elementary hyperbolic geometry shows that the unique geodesic orthogonal to the two hyperbolic hyperplanes containing $F_{s,1}'$ and $F_{s,2}'$ must come very close to $x$.  In particular, since $\frac{s}{\delta} \to 1$ as $s \to \infty$, the distance between these hyperplanes is asymptotic to $2s$ (i.e. the ratio of this distance to $2s$ tends to $1$).

Let $F_{r,1},F_{r,2} \in \mathcal F(Q_r)$ be a pair of faces of $Q_r$ so that $\rho_s(\tau_{F_{r,i}}) = \tau_{F_{s,i}'}$ for $i = 1,2$.
Observe that by Lemma~\ref{L:finite Q-hull}, we have $Q_r \subset B_{r+R}(x)$.
Consequently, the distance between the hyperplanes containing $F_{r,1}$ and $F_{r,2}$ is at most
\[ 2(r+R) = 2\left(\frac{\delta}{K} + R\right) = 2 \left(\frac{1 + s + R}{K} + R \right).\]
In particular, as $s \to \infty$, $K \to 1$ and so this upper bound is also asymptotic to $2s$.  

Now let $\gamma_s = \tau_{F_{r,1}}\circ \tau_{F_{r,2}}$ so that $\rho_s(\gamma_s) = \tau_{F_{s,1}'}\circ \tau_{F_{s,2}'}$.  The translation length of the product of two reflections (in disjoint hyperplanes) is twice the distance between the hyperplanes.  So, $\ell(\rho_s(\gamma_s))$ is asymptotic to $4s$ and $\ell(\gamma_s)$ is bounded above by a quantity asymptotic to $4s$.  Since
\[ \ell(\rho(\gamma_s)) \leq K\ell(\gamma_s) \leq \ell(\gamma_s),\]
it follows that 
\[ 1 \leq \liminf_{s \to \infty} \frac{\ell(\gamma_s)}{\ell(\rho_s(\gamma_s))} \leq  \limsup_{s \to \infty} \frac{\ell(\gamma_s)}{\ell(\rho_s(\gamma_s))} \leq \limsup_{s \to \infty} \frac{4s}{4s} = 1.\]
Consequently,
\[ \lim_{s \to \infty} \frac{\ell(\gamma_s)}{\ell(\rho_s(\gamma_s))} = 1.\]
The theorem follows now by choosing any sequence $s_n \to \infty$, setting $\Gamma_n' = \Gamma_{s_n}'$ and $\Gamma_n = \Gamma_{r(s_n)}$.
\end{proof}

As an application of Theorem~\ref{T:main RACG} and Theorem~\ref{T:Tholozan volume}, we have the following.

\begin{corollary} \label{C:arbitrarily large image}  For any right-angled rotation group $\Gamma < \SO_0(d,1)$, for $d \leq 4$, and integer $n \geq 1$, there exists a finite index subgroup $\Gamma_n < \Gamma$, and $n$ representations $\rho_1,\ldots,\rho_n \in \Dom(\Gamma_n,\SO_0(d,1))$ such that $\Vol(\rho_i) \neq \Vol(\rho_k)$ for all $i \neq k$.  Consequently, if $j \colon \Gamma_n \to \SO_0(d,1)$ is the inclusion, then $\Vol(j \times \rho_i) \neq \Vol(j \times \rho_k)$ for all $i \neq k$.
\end{corollary}
\begin{proof}  The subgroups $\Gamma_n$ are defined recursively by iteratively applying Theorem~\ref{T:main RACG}.  Specifically, starting with the pair of right-angled rotation groups $\Gamma,\Gamma$ (i.e.~the {\em same} rotation group), let $\Gamma_0 = \Gamma$, and the theorem provides a finite index subgroup $\Gamma_1 < \Gamma_0$ (infinitely many, in fact, but we just pick one) and a homomorphism $\rho_1^1 \colon \Gamma_1 \to \Gamma_0 < \SO_0(d,1)$ which is strictly dominated by the inclusion $\Gamma_1 \to \SO_0(d,1)$.  Next, apply Theorem~\ref{T:main RACG} again, this time to the pair $\Gamma_1,\Gamma_1$, to produce a finite index subgroup $\Gamma_2 < \Gamma_1$, and $\rho_1^2 \colon \Gamma_2 \to \Gamma_1 < \SO_0(d,1)$ strictly dominated by the inclusion $\Gamma_2 \to \SO_0(d,1)$.  Now we can produce a second representation $\rho_2^2$ of $\Gamma_2$ as the composition $\rho_2^2 = \rho_1^1 \circ \rho_1^2 \colon \Gamma_2 \to \Gamma_0 < \SO_0(d,1)$ which is also strictly dominated by the inclusion $\Gamma_2 \to \SO_0(d,1)$.  Furthermore, observe that
\[ \Vol(\rho^2_2) = \vol(M_{\Gamma_0}) \neq \vol(M_{\Gamma_1}) = \vol(\rho_1^2).\]
We can iterate this construction, producing a descending chain of finite index subgroups $\Gamma = \Gamma_0 > \Gamma_1 > \Gamma_2 > \cdots > \Gamma_n > \cdots$ and representations
\[ \rho_1^n \colon \Gamma_n \to \Gamma_{n-1} < \SO_0(d,1)\]
and for $i = 2,\ldots,n$,
\[ \rho_i^n = \rho_1^{n-i+1} \circ \rho_1^{n-i+2} \circ \circ \cdots \circ \rho_1^{n-1} \circ \rho_1^n \colon \Gamma_n \to \Gamma_{n-i} < \SO_0(d,1).\]
These $\rho_i^n$ are strictly dominated by the inclusion $\Gamma_n \to \SO_0(d,1)$, and for all $i \neq k$, 
\[ \Vol(\rho_i^n) = \Vol(M_{\Gamma_{n-i}}) \neq \Vol(M_{\Gamma_{n-k}}) = \Vol(\rho_k^n),\]
as required.
The last statement follows from Theorem~\ref{T:Tholozan volume}.
\end{proof}

\section{Tetrahedral examples} \label{S:tetrahedron}

Agol's construction makes use of the group generated by reflections in a pair of hyperbolic tetrahedra.  To describe his construction, as well as the noncocompact examples, we let $T  = T(a_1,a_2,a_3;b_1,b_2,b_3)$ be the (possibly ideal) hyperbolic tetrahedron with (interior) dihedral angles $\frac{\pi}{a_i},\frac{\pi}{b_i}$ at the six edges.  These are listed with the following convention: fixing a vertex $v$ of the tetrahedron, $\frac{\pi}{a_1},\frac{\pi}{a_2},\frac{\pi}{a_3}$ are the dihedral angles at the three edges {\em not} adjacent to $v$, while $\frac{\pi}{b_1},\frac{\pi}{b_2},\frac{\pi}{b_3}$ are the dihedral angles adjacent to $v$, with the $b_i$ labels opposite $a_i$, for $i=1,2,3$.

The group generated by reflections in the faces of $T$ is a discrete group when the $a_i$ and $b_i$ are all positive integers.  As shown in \cite{lanner,vinberg,Thurston-notes}, this occurs for exactly nine compact tetrahedra and 23 noncompact tetrahedra; see~\cite{brunneretal}, for example, for a complete, concise list.  Of particular interest to us are the five tetrahedra
\[ T_j = T(2,3,j;2,3,5) \mbox{ for } j = 2,3,4,5,6. \]
The tetrahedron $T_6$ has two ideal vertices, while the other four are compact.  (Note: the index $j$ is our own notation, and does not conform to the notation in \cite{brunneretal}.)

We note that all five tetrahedra may be realized in $\mathbb H^3$ with a common vertex $v$, in which the adjacent dihedral angles are $\frac{\pi}2$, $\frac{\pi}{3}$, and $\frac{\pi}{5}$.  In fact, the tetrahedra so obtained are nested
\[ T_2 \subset T_3 \subset T_4 \subset T_5 \subset T_6. \]
For each $j = 2,3,4,5,6$, the three $2$--dimensional faces $F_1^j,F_2^j,F_3^j$ of $T_j$ adjacent to $v$ are contained in three common hyperbolic planes $P_1,P_2,P_3$, while the fourth faces $F_4^j$ are contained in a different planes $P_4^j$, $j = 2,3,4,5,6$.  For concreteness, we suppose the labelings are such that the angles between $P_i$ and $P_4^j$, for $i=1,2,3$ are $\frac{\pi}{2},\frac{\pi}{3},\frac{\pi}{j}$, respectively.  See Figure~\ref{F:nested-tet} where $T_2,T_3,T_4,T_6$ are illustrated in the upper half space model.

\begin{figure}[htb]
\begin{center}
\begin{tikzpicture}[xscale=1,yscale=1]
\node at (0,0) {\includegraphics[scale=.35]{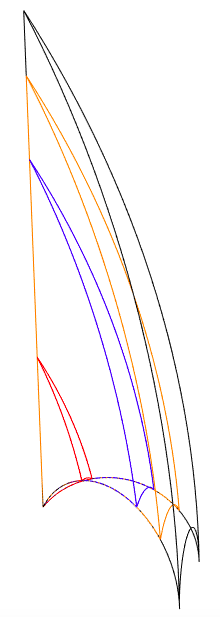}};
\end{tikzpicture}
\caption{\label{F:nested-tet} The nested hyperbolic tetrahedra $T_2 \subset T_3 \subset T_4 \subset T_6$.}
\end{center}
\end{figure}

For each $j = 2, 3,4, 5,6$, let $\bar \Gamma_j$ denote the group generated by reflections in the faces $F_i^j$ of $T_j$, and $\Gamma_j < \bar \Gamma_j$, the index two rotation subgroup.  The group $\bar \Gamma_j$ has a presentation
\[ \bar \Gamma_j = \langle \tau_1,\tau_2,\tau_3,\tau_4 \mid (\tau_i\tau_k)^{e_{ik}} \rangle.\]
Here $\tau_i$ is the reflection in the plane $P_i$, for $i = 1,2,3$ and the reflection in the plane $P_4^j$ for $i = 4$.  The $e_{ik}$ in the relations are as follows:  for each $i$, $e^{ii} =1$ (since $\tau_i$ has order $2$), and for $i \neq k$, $\frac{\pi}{e_{ik}}$ is the dihedral angle on the unique edge adjacent to the faces in which $\tau_i$ and $\tau_k$ are reflections (indeed, $\tau_i\tau_k$ is rotation of order $e_{ik}$ about the geodesic line containing the edge).

\begin{proposition} \label{P:tetrahedron homs} If $j|j'$, then $\rho(\tau_i) = \tau_i$, for each $i$, defines a surjective homomorphism $\rho_{jj'} \colon \bar \Gamma_{j'} \to \bar \Gamma_j$.
\end{proposition}
\begin{proof}  We need only verify that the $\rho$--images of the relations in $\bar \Gamma_{j'}$ are satisfied in $\bar \Gamma_j$.  The only nontrivial relation we must check is that $\rho(\tau_3)\rho(\tau_4)^{j'} = 1$.  For clarity, write $\rho(\tau_i) = \hat \tau_i$, so that $(\hat \tau_3 \hat \tau_4)^j = 1$.  Then since $j | j'$ we can write $j' = kj$ and hence
\[ (\rho(\tau_3)\rho(\tau_4))^{j'} = (\hat \tau_3 \hat \tau_4)^{j'} = ((\hat \tau_3 \hat \tau_4)^j)^k = 1^k = 1.\]
\end{proof}

\subsection{Compact examples}

Agol's original example can now be stated as follows.
\begin{theorem}[Agol] \label{T:Agol}  There exists a $K$--Lipschitz, degree $1$ map $f \colon M_{\Gamma_4} \to M_{\Gamma_2}$ for some $K < 1$.
\end{theorem}
\begin{proof}  We view $T_2 \subset T_4$ as above, inside the Klein model $\mathbb K^3$ of $\mathbb H^3$ with the vertex $v$ at the origin.  Let $\ell_1,\ell_2,\ell_3$ denote the lines in $\mathbb R^3$ through the origin containing the three edges of both $T_2$ and $T_4$ adjacent to $v$; see Figure~\ref{F:nested-tet-klein}.

\begin{figure}[htb]
\begin{center}
\begin{tikzpicture}[xscale=1,yscale=1]
\node at (0,0) {\includegraphics[scale=.35]{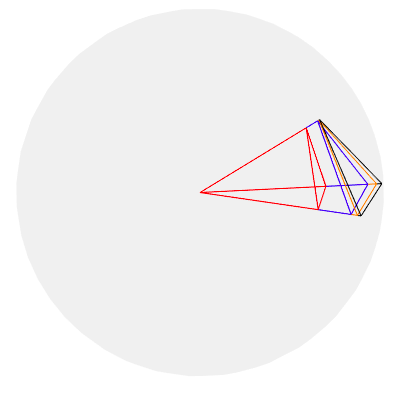}};
\end{tikzpicture}
\caption{\label{F:nested-tet-klein} The nested hyperbolic tetrahedra $T_2 \subset T_3 \subset T_4 \subset T_6$ in $\mathbb K^3$.}
\end{center}
\end{figure}

There is a unique linear map $h \colon \mathbb R^3 \to \mathbb R^3$ with $h(\ell_i) = \ell_i$, for  $i = 1,2,3$, and $h(T_4) =  T_2$.  If we can show that $h(\mathbb K^3)$ is contained in a ball of Euclidean radius $K$ centered at $0$, for some $K < 1$, then by Corollary~\ref{C:explicit birkhoff}, $h$ will be a $K$--Lipschitz map from $\mathbb K^3$ to itself. In particular, the restriction is a $K$--Lipschitz map $T_4 \to T_2$.  As in the case of right-angled reflection groups, this map extends to $\rho$--equivariant map which descends to a degree $1$ map $M_{\Gamma_4} \to M_{\Gamma_2}$, as required (see the proof of Theorem~\ref{T:main 1 simpler}).

To construct the required linear map $h$, we explicitly find $T_2$ and $T_4$ as follows.  Let $\mathbb R^{3,1}$ denote Minkowski space---that is, $\mathbb R^4$ with the quadratic form $Q(x,y,z,w) = x^2 + y^2 + z^2-w^2$ and associated bilinear form $B$.  Then $\mathbb K^3$ sits in $\mathbb R^{3,1}$ via the embedding $\mathbb K^3 \subset \mathbb R^3 = \{w = 1\} \subset \mathbb R^{3,1}$ and the hyperbolic planes in $\mathbb K^3$ are intersections with $\mathbb K^3$ of $Q$--orthogonal complements of vectors $u \in \mathbb R^{3,1}$ with $Q(u) = 1$.  For $Q(u) = 1$, write $u^\perp$ for the hyperbolic plane so defined.  For any two vectors $u,v$ with $Q(u) = Q(v) = 1$, and $|B(u,v)| < 1$, the cosine of the angle of intersection between $u^\perp$ and $v^\perp$ is exactly $B(u,v)$.  More precisely, $u,v$ are naturally identified with normal vectors to $u^\perp$ and $v^\perp$, respectively, and $B(u,v)$ is the cosine of the angle of intersection of the normal vectors;  see, for example~\cite{Thurston-notes}.

We can explicitly compute the vectors defining the planes $P_1,P_2,P_3,P_4^j$, for $j = 2,4$, and we get
\[ P_1 = (0,0,1,0)^\perp, P_2 = (0,1,0,0)^\perp, P_3 = \left( \frac{\sqrt{6 - 2 \sqrt{5}}}4,\frac{-1}2,\frac{-1 - \sqrt{5}}4,0 \right)^\perp, \]
\[ P_4^2 = \left(\frac{-1-\sqrt{5}}{2\sqrt{6 - 2 \sqrt{5}}},0,\frac{-1}2,\frac{-\sqrt{1 +3 \sqrt{5}}}{2\sqrt{2}} \right)^\perp, \]
and
\[ P_4^4 = \left(\frac{-1-2\sqrt{2} - \sqrt{5}}{2\sqrt{6 - 2 \sqrt{5}}},\frac{-1}{\sqrt{2}},\frac{-1}2,-\sqrt{\frac{2 + \sqrt{2} + \sqrt{5} + \sqrt{10}}{6 - 2 \sqrt{5}}} \right)^\perp. \]
We use linear algebra to find the intersections of the $Q$--orthogonal complements, and hence the vertices of the tetrahedra.  These are approximately
\[ \{(0,0,0),(.749871, 0, 0),(.749871, .463446, 0),(.654396, 0,.249957) \}, \]
for $T_2$ and
\[\{(0,0,0),(.979008, 0, 0), (.830974, .51357, 0),(.908295, 0,.346938) \}, \]
for $T_4$.  (Exact solutions are possible as the above computations suggest, but we describe the approximates as the expressions become increasingly messy).
From this we easily write down the linear transformation $h \colon \mathbb R^3 \to \mathbb R^3$ sending the vertices of $T_4$ to the vertices of $T_2$, and compute that its largest singular value is appproximately $.976184 < 1$.  Consequently, $h$ is $.977$--Lipschitz.
\end{proof}

\subsection{Noncompact examples and hyperbolic Dehn filling.}

Here we consider the homomorphism $\rho \colon \bar \Gamma_6 \to \bar \Gamma_2$ guaranteed by Proposition~\ref{P:tetrahedron homs}.
\begin{theorem}  \label{T:noncompact examples}  There exists a $\rho$--equivariant $K$--Lipschitz map $\widetilde f \colon \mathbb H^3 \to \mathbb H^3$ for some $K < 1$ that descends to a $K$--Lipschitz  
\[ f \colon M_{\Gamma_6} \to M_{\Gamma_2}. \]
Furthermore, $f$ sends some neighborhood of each cusp of $M_{\Gamma_6}$ to a single point, and
\[ \int_{M_{\Gamma_6}} f^* d\vol_{\mathbb H} = \Vol(M_{\Gamma_2}). \]
\end{theorem}
\begin{proof}
We make use of the set-up and construction from the proof of Theorem~\ref{T:Agol}.  Recall that we constructed a $K$--Lipschitz diffeomorphism $h \colon T_4 \to T_2$ (which determined the map $M_{\Gamma_4} \to M_{\Gamma_2}$).   Let $\pi \colon \mathbb H^3 \to T_4$ denote the ($1$--Lipschitz) closest point projection (see Lemma~\ref{L:exponential contraction}).   Composing with $h$, we obtain a $K$--Lipschitz map we denote $\widetilde f^\circ = h \circ \pi \colon T_6 \to T_2$.

As in Section~\ref{S:Lebesgue numbers}, we note that the preimage in $\mathbb H^3$ of the interior of a $2$--dimensional face $F_i^4$ of $T_4$ under $\pi$ is the interior of one of the regions bounded by hyperbolic planes containing the edges of $F_i^4$ and orthogonal to $F_i^4$.  The preimage of an edge is similarly described as a region bounded by portions of hyperbolic planes containing the edge and portions of planes orthogonal to the edge through the endpoints.  The rest of $\mathbb H^3$ is mapped to the vertices.  From this, and the fact that all dihedral angles of $T_4$ are at most $\frac{\pi}{2}$, it follows that $\pi$ sends the closure of $F_i^6$ to the closure of $F_i^4$, for each $i = 1,2,3,4$.
Furthermore, we see that for each ideal vertex, $\pi$ sends some neighborhood of that vertex to the corresponding vertex of $T_4$.  Composing with $h$, we see that $\widetilde f^\circ$ also has this property.  Furthermore, observe that on the preimage of the interior of $T_2$ (which is exactly the interior of $T_4 \subset T_6$), $\widetilde f^\circ$ is a diffeomorphism and hence
\[ \int_{T_6} \widetilde f^{\circ *} d\vol_{\mathbb H} = \Vol(T_2).\]

Since $\rho \colon \bar \Gamma_6 \to \bar \Gamma_2$ sends reflection in $F_i^6$ to the reflection in $F_i^2$ and $\widetilde f^\circ(\bar F_i^6) = \bar F_i^2$, it follows that we can $\rho$--equivariantly extend $\widetilde f^\circ$ to a map $\widetilde f \colon \mathbb H^3 \to \mathbb H^3$.  As in the right-angled reflection group case, $M_{\Gamma_6}$ and $M_{\Gamma_2}$ are obtained by doubling $T_6$ and $T_2$ over their boundaries, and the descent $f \colon M_{\Gamma_6} \to M_{\Gamma_2}$ of $\widetilde f$ sends each copy of $T_6$ to a copy of $T_2$ by an orientation preserving map and
\[ \int_{M_{\Gamma_6}} f^* d\vol_{\mathbb H} = 2 \Vol(T_2) = \Vol(M_{\Gamma_2}). \]
Finally, we note that since $\widetilde f^\circ$ sends a neighborhood of the each ideal vertex to a vertex of $T_2$, it follows that $f$ sends a neighborhood of each cusp of $M_{\Gamma_6}$ to a point of $M_{\Gamma_2}$. 
\end{proof}

\begin{remark} Note that since $3|6$, Proposition~\ref{P:tetrahedron homs} provides a homomorphism $\rho \colon \bar \Gamma_6 \to \bar \Gamma_3$.  Although we do not carry out the details here, the proof of Theorem~\ref{T:noncompact examples} can be modified so that the same conclusion holds for this homomorphism.  The main point is to replace the linear diffeomorphism $T_4 \to T_2$ with a linear diffeomorphism $T' \to T_3$, where $T'$ is obtained by a $\lambda$--homothety of $T_3$, for $\lambda > 1$, but very close to $1$.   Then the linear map $T' \to T_3$ is a homothety by $1/\lambda$ which is $\frac{1}\lambda$--Lipschitz by Corollary~\ref{C:explicit birkhoff}.  The map $T_6 \to T_3$ is just the composition of the closest point projection $T_6 \to T'$, composed with the $\frac{1}\lambda$--homothety $T' \to T_3$.
\end{remark}

Now suppose that $\Lambda \subset \Gamma_6$ is a torsion free subgroup of index $m < \infty$, and $p \colon M_\Lambda \to M_{\Gamma_6}$ the corresponding orbifold covering map.  Since the covering map is a local isometry, the composition $f \circ p \colon M_\Lambda \to M_{\Gamma_2}$ is $K$--Lipschitz  for some $K < 1$.  Suppose $M_\Lambda$ has $k$ cusps, and let $C_1,\ldots,C_k \subset M_\Lambda$ be torus cusp neighborhoods of the $k$ cusps such that $f \circ p$ is locally constant on $C = C_1 \cup \ldots \cup C_k$ (we choose torus cusps $C_i$ which map into the cusp neighborhoods of $M_{\Gamma_6}$ on which $f$ is locally constant).

Given primitive elements $\alpha_i$ in $\pi_1(C_i)$, we can remove the cusp neighborhood $C_i$ and glue a solid torus $V_i$, identifying the torus boundary of $V_i$ with the exposed torus boundary component of $M_\Lambda - C$ in such a way that the primitive element $\alpha_i$ is trivial in $V_i$.  If we write $\alpha = (\alpha_1,\ldots,\alpha_n)$, then the resulting manifold is called the {\bf $\alpha$--Dehn filling} of $M_\Lambda$, and $\alpha$ is called the {\bf Dehn-filling coefficent}.   Thurston proved that there exists a finite set $A_i \subset \pi_1(C_i)$ so that if for each $i$, $\alpha_i \not \in A_i$, then for $\alpha = (\alpha_1,\ldots,\alpha_n)$, the $\alpha$--Dehn filling of $M_\Lambda$ is again a hyperbolic $3$--manifold, which we denote $M_{\Lambda_\alpha}$.  Note that the inclusion $M_\Lambda-C \to M_\Lambda$ induces an isomorphism on fundamental groups, while the inclusion $M_\Lambda - C \to M_{\Lambda_\alpha}$ induces a surjection on fundamental groups.  Thus we have a surjective homomorphism $\Lambda \to \Lambda_\alpha$.  Indeed, this homomorphism is simply the quotient homomorphism by the normal closure of $\alpha_1,\ldots,\alpha_k$.  More is true, and we state here a version that we will be most useful for us; see \cite{Thurston-notes,NeumannZagier,HodgsonKerckhoff}

\begin{theorem} [Thurston] \label{T:Dehn filling}  With the notation as above, suppose that
\[ \{\alpha^n = (\alpha_1^n,\ldots,\alpha_k^n)\}_{n=1}^\infty \]
is any sequence of Dehn fillings coefficients such that for each $i$, $\{\alpha_i^n\}_{n=1}^\infty$ are all distinct.  Then for any $\epsilon > 0$, there exists $N > 0$ so that for all $n \geq N$, the inclusion $M_\Lambda - C$ into $M_{\Lambda_\alpha}$ is a $(1+\epsilon)$--biLipschitz embedding.
\end{theorem}
Here we note that $M_\Lambda - C$ is given the induced path metric from $M_\Lambda$ on the domain, and the induced path metric from $M_{\Lambda_\alpha}$ on the target.

From this, and Theorem~\ref{T:noncompact examples}, we easily obtain the following.

\begin{theorem} \label{T:dehn filling examples} Suppose that $\Lambda< \Gamma_6$ is a torsion-free subgroup of index $m < \infty$, and 
\[ \{\alpha^n = (\alpha_1^n,\ldots,\alpha_k^n)\}_{n=1}^\infty \]
is any sequence of Dehn fillings coefficients for $M_\Lambda$ such that for each $i$, $\{\alpha_i^n\}_{n=1}^\infty$ are all distinct.  Then there exists $K' < 1$ and $N > 0$ so that for all $n \geq N$, there exists a degree $m$, $K'$--Lipschitz map $f_n \colon M_{\Lambda_{\alpha_n}} \to M_{\Gamma_2}$.  If we write $\rho_n = f_{n*} \colon \Lambda_{\alpha_n} \to \Gamma_2$ and $j_n \colon \Lambda_{\alpha_n} \to \SO_0(3,1)$ for the inclusion, then
\[ \Vol(j_n \times \rho_n) = {\bf V}_3(\vol(M_{\Lambda_{\alpha_n}}) - m \vol(M_{\Gamma_2})). \]
\end{theorem}
\begin{proof}  Let $f \colon M_{\Gamma_6} \to M_{\Gamma_2}$ be the $K$--Lipschitz map from Theorem~\ref{T:noncompact examples},  $p \colon M_\Lambda \to M_{\Gamma_6}$ denote the covering map, and $f \circ p \colon M_\Lambda \to M_{\Gamma_2}$ the resulting $K$--Lipschitz composition.  Let $C_1,\ldots,C_k$ denote the cusp neighborhoods on which $f\circ p$ is locally constant, $C = C_1 \cup \cdots \cup C_k$, and let $N > 0$ from  Theorem~\ref{T:Dehn filling} be such that the inclusion $M_\Lambda - C \to M_{\Lambda_{\alpha_n}}$ is $\frac{2}{K+1}$--biLipschitz for all $n \geq N$. (Note that since $0 < K < 1$, we have $\frac{2}{K+1} > 1$).

Since $f \circ p$ is constant on each $C_i$, the boundary of each $C_i$ is sent to a point.  Consequently, we can extend this map (by a constant) over any filling solid torus and denote it $f_n$:
\[ \xymatrix{ M_\Lambda - C \ar[r] \ar[d] & M_{\Gamma_2} \\ M_{\Lambda_{\alpha_n}} \ar[ru]_{f_n}}
\]
On the complement of the filling solid tori, $f_n$ is the composition of the inverse of the $\frac{2}{K+1}$--biLipschitz map and the $K$--Lipschitz map, and is consequently $K'$--Lipschitz, where $K' = \frac{2K}{K+1} < 1$.  On the filling solid tori, $f_n$ is constant, so it is obvious $K'$--Lipschitz.  Thus, $f_n$ is $K'$--Lipschitz, as required.

Since $f \circ p$ is constant on each $C_i$, we have
\[ \int_{M_\Lambda - C} (f \circ p)^* d\vol_{\mathbb H} =\int_{M_\Lambda} (f \circ p)^* d\vol_{\mathbb H}  = m \int_{M_{\Gamma_6}} f^* d\vol_{\mathbb H} = m \vol(M_{\Gamma_2}).\]
From this and since $f_n$ is constant on the solid tori, we have
\[ \int_{M_{\Lambda_{\alpha_n}}} f_n^* d \vol_{\mathbb H} =  \int_{M_\Lambda - C} (f \circ p)^* d\vol_{\mathbb H} = m \vol(M_{\Gamma_2}).\]
Thus $f_n$ has degree $m$, and the volume expression follows from Theorem~\ref{T:Tholozan volume}.
\end{proof}

\bibliographystyle{amsplain}

\providecommand{\bysame}{\leavevmode\hbox to3em{\hrulefill}\thinspace}
\providecommand{\MR}{\relax\ifhmode\unskip\space\fi MR }
\providecommand{\MRhref}[2]{%
  \href{http://www.ams.org/mathscinet-getitem?mr=#1}{#2}
}
\providecommand{\href}[2]{#2}


\noindent Department of Mathematics and Computer Science,\\
Eastern Illinois University,\\
600 Lincoln Avenue,\\
Charleston, IL 61920.\\
Email: gslakeland@eiu.edu

\noindent Department of Mathematics,\\
University of Illinois at Urbana-Champaign,\\
1409 W. Green St,\\
Urbana, IL 61801.\\
Email: clein@math.uiuc.edu

\end{document}